\documentclass[12pt]{article}
\usepackage{amsmath,amssymb,amsbsy,amsfonts,amsthm,latexsym,
            amsopn,amstext,amsxtra,euscript,amscd}
\begin{document}
\bibliographystyle{plain}

\newfont{\teneufm}{eufm10}
\newfont{\seveneufm}{eufm7}
\newfont{\fiveeufm}{eufm5}
%
%
\newfam\eufmfam
              \textfont\eufmfam=\teneufm \scriptfont\eufmfam=\seveneufm
              \scriptscriptfont\eufmfam=\fiveeufm
%
%
\def\frak#1{{\fam\eufmfam\relax#1}}
%


\newtheorem{theorem}{Theorem}
\newtheorem{lemma}[theorem]{Lemma}
\newtheorem{claim}[theorem]{Claim}
\newtheorem{cor}[theorem]{Corollary}
\newtheorem{prop}[theorem]{Proposition}
\newtheorem{definition}[theorem]{Definition}
\newtheorem{remark}[theorem]{Remark}
\newtheorem{question}[theorem]{Open Question}

\def\qed{\ifmmode
\squareforqed\else{\unskip\nobreak\hfil
\penalty50\hskip1em\null\nobreak\hfil\squareforqed
\parfillskip=0pt\finalhyphendemerits=0\endgraf}\fi}

\def\squareforqed{\hbox{\rlap{$\sqcap$}$\sqcup$}}

\def\cA{{\mathcal A}}
\def\cB{{\mathcal B}}
\def\cC{{\mathcal C}}
\def\cD{{\mathcal D}}
\def\cE{{\mathcal E}}
\def\cF{{\mathcal F}}
\def\cG{{\mathcal G}}
\def\cH{{\mathcal H}}
\def\cI{{\mathcal I}}
\def\cJ{{\mathcal J}}
\def\cK{{\mathcal K}}
\def\cL{{\mathcal L}}
\def\cM{{\mathcal M}}
\def\cN{{\mathcal N}}
\def\cO{{\mathcal O}}
\def\cP{{\mathcal P}}
\def\cQ{{\mathcal Q}}
\def\cR{{\mathcal R}}
\def\cS{{\mathcal S}}
\def\cT{{\mathcal T}}
\def\cU{{\mathcal U}}
\def\cV{{\mathcal V}}
\def\cW{{\mathcal W}}
\def\cX{{\mathcal X}}
\def\cY{{\mathcal Y}}
\def\cZ{{\mathcal Z}}
\newcommand{\rmod}[1]{\: \mbox{mod}\: #1}

\def\tcN{\cN^\mathbf{c}}

\def\Tr{{\mathrm{Tr}}}

\def\mand{\qquad \mbox{and} \qquad}
\renewcommand{\vec}[1]{\mathbf{#1}}

\def\eqref#1{(\ref{#1})}


\newcommand{\ignore}[1]{}

\hyphenation{re-pub-lished}

\parskip 1.5 mm
\def\lln{{\mathrm Lnln}}
\def\Res{\mathrm{Res}\,}

\def\F{{\mathbb F}}
\def\L{{\mathbb L}}
\def\Q{{\mathbb Q}}
\def\R{{\mathbb R}}
\def\K{{\mathbb K}}
\def\Z{{\mathbb Z}}

\def\Fp{\F_p}
\def\fp{\Fp^*}
\def\Fq{\F_q}
\def\ff{\F_2}
\def\ffn{\F_{2^n}}

\def\Zm{\Z_m}
\def \Um{{\mathcal U}_m}

\def \Bf{\frak B}

\def\Km{\cK_\mu}

\def\va {{\mathbf a}}
\def \vb {{\mathbf b}}
\def \vc {{\mathbf c}}
\def\vx{{\mathbf x}}
\def \vr {{\mathbf r}}
\def \vv {{\mathbf v}}
\def\vu{{\mathbf u}}
\def \vw{{\mathbf w}}
\def \vz {{\mathbfz}}

\def\\{\cr}
\def\({\left(}
\def\){\right)}
\def\fl#1{\left\lfloor#1\right\rfloor}
\def\rf#1{\left\lceil#1\right\rceil}

\def\flq#1{{\left\lfloor#1\right\rfloor}_q}
\def\flp#1{{\left\lfloor#1\right\rfloor}_p}
\def\flm#1{{\left\lfloor#1\right\rfloor}_m}

\def\Al{{\sl Alice}}
\def\Bob{{\sl Bob}}

\def\Or{{\mathcal O}}

\def\inv#1{\mbox{\rm{inv}}\,#1}
\def\invM#1{\mbox{\rm{inv}}_M\,#1}
\def\invp#1{\mbox{\rm{inv}}_p\,#1}

\def\Ln#1{\mbox{\rm{Ln}}\,#1}

\def \nd {\,|\hspace{-1.2mm}/\,}

\def\ord{\mu}

\def\E{\mathbf{E}}

\def\Cl{{\mathrm {Cl}}}

\def\epp{\mbox{\bf{e}}_{p-1}}
\def\ep{\mbox{\bf{e}}_p}
\def\eq{\mbox{\bf{e}}_q}

\def\bm{\bf{m}}

\newcommand{\floor}[1]{\lfloor {#1} \rfloor}

\newcommand{\comm}[1]{\marginpar{%
\vskip-\baselineskip 
\raggedright\footnotesize
\itshape\hrule\smallskip#1\par\smallskip\hrule}}

\def\rem{{\mathrm{\,rem\,}}}
\def\dist {{\mathrm{\,dist\,}}}
\def\etal{{\it et al.}}
\def\ie{{\it i.e. }}
\def\veps{{\varepsilon}}
\def\eps{{\eta}}

\def\ind#1{{\mathrm {ind}}\,#1}
               \def \MSB{{\mathrm{MSB}}}
\newcommand{\abs}[1]{\left| #1 \right|}

\title{On Fully Split Lacunary Polynomials in Finite Fields}

\author {{\sc Khodakhast Bibak $^1$ and Igor E.  Shparlinski $^2$}\\ \\
 $^1$ Department of Combinatorics and Optimization\\University of Waterloo\\
              Waterloo, Ontario, Canada N2L 3G1\\
              {\tt kbibak@uwaterloo.ca}
\and
$^2$ Department of Computing\\ Macquarie University\\
             Sydney, NSW 2109, Australia\\
             {\tt igor@comp.mq.edu.au}
              }

\maketitle

\begin{abstract}
We estimate the number of possible types degree patterns
of $k$-lacunary polynomials of degree $t < p$ which split
completely modulo $p$. The result is based on a combination
of a bound on the number of zeros of lacunary polynomials with
some graph theory arguments.
\end{abstract}


\section{Introduction}

Zeros and factorisations of lacunary polynomials, that is,
polynomials of high degree with relatively small number of
non-zero coefficients, has always been a subject of active
investigation, see~\cite{CTV,FGS,Len1,Len2,Schin} and references
therein. We say that a polynomial $f$ over a field $\K$ is
$k$-lacunary if it has at most $k+1$ non-zero coefficients,
including a non-zero constant term, that is, if $f(0) \ne 0$ and
\begin{equation}
\label{eq:LacPoly}
f(X) = a_0 + a_1X^{t_1} + \ldots + a_kX^{t_k} \in \K[X]
\end{equation}
for some positive integers $t_1 < \ldots <t_k$.

For example, a classical result of Descartes asserts that a
$k$-lacunary polynomial $f\in  \R[X]$ may have at most $2k$ real
roots. Furthermore, Lenstra~\cite{Len2} has shown that for an
algebraic  number field $\K$ of degree $m$ over $\Q$ and a
$k$-lacunary polynomial $f\in \K[X]$, the product $g$ of all
irreducible divisors $h\mid f$ of degree at most $\deg h \le d$
 is of degree
$$
\deg g  = O\(k^2 2^{md} md \log (2mdk)\).
$$

Schinzel~\cite{Schin} has obtained a series of statistical results
about the number of $k$-lacunary irreducible polynomials with
prescribed coefficients. In particular,
by~\cite[Corollary~2]{Schin}, for any algebraic numbers $a_0,
\ldots,a_k$ there are at most $O\(T^{\fl{(k+1)/2}}\)$ $k$-tuples
of integers
\begin{equation}
\label{eq:k tuples}
\vec{t} = (t_1,\ldots, t_k),
\qquad  1 \le t_1 < \ldots <t_k,
\end{equation}
with  $t_k \le T$ and such that the  largest non-cyclotomic factor
(that is, a factor which does not have roots that are roots of
unity) of the $k$-lacunary polynomial~\eqref{eq:LacPoly} is
reducible over $\K = \Q\(a_1/a_0, \ldots, a_k/a_0\)$.

Here we consider a related question about estimating the number
$N_k(p,t)$ of $k$-tuples~\eqref{eq:k tuples}  such that there is a
$k$-lacunary polynomial  of the form~\eqref{eq:LacPoly} of degree $t_k = t$ over the finite field
$\K =\F_p$ of $p$ elements, where $p$ is a prime,
that fully splits over $\F_p$.

\begin{theorem}
\label{thm:split p} If a positive integer $k$ is fixed then for
any prime $p$ and  positive integer $t < p$, we have,
$$
N_k(p,t) \le t^{k - k\rf{(k-3)/2}-1} p^{(k-1)\rf{(k-3)/2}+o(1)}
$$
as $p\to \infty$.
\end{theorem}

Clearly, Theorem~\ref{thm:split p} is nontrivial only for $k > 3$
and for
\begin{equation}
\label{eq:large t}
t > p^{1-1/k +\varepsilon},
\end{equation}
with some fixed $\varepsilon>0$. Furthermore, for $t \gg p$ we
obtain the bound
$$
N_k(p,t) \le t^{\rf{k/2} +1 +o(1)}.
$$

Our result is based on a rather unusual combination of two techniques:
a bound on the number of zeros of lacunary polynomials
(see Section~\ref{sec:poly})
and a bound on the so-called domination number of a graph
(see Section~\ref{sec:dom}).

Throughout the paper,
the implied constants in the symbols `$O$', `$\ll$'
and `$\gg$'   may depend on $k$  (we recall that
the notations $U \ll V$ and $V \gg  U$ is
equivalent to   $U = O(V)$).

\section{Zeros of Lacunary Polynomials}
\label{sec:poly}

We need the following estimate from~\cite{CFKLLS} on the number of zeros of lacunary
polynomials over $\F_p$.

\begin{lemma}
\label{lem:SprEq Zeros}
For
$k +1\ge 2$
elements $a_0, a_1, \ldots\,, a_k \in \F_p^*$ and
integers $0 = t_0 < t_1 < \ldots < t_k<p$,
the number of solutions  $Q$  to the equation
$$
\sum_{i=0}^k a_ix^{t_i} = 0, \qquad x \in \F_p^*,
$$
with $t_0 = 0$, satisfies
$$
Q \le 2  p^{1 - 1/k} D^{1/k} + O(p^{1 - 2/k} D^{2/k}),
$$
where
$$
D = \min_{0 \le i \le k} \max_{j \ne i} \gcd(t_j - t_i, p-1).
$$

\end{lemma}

\begin{lemma}
\label{lem:SprEq Mult}
For
$k+1 \ge 2$
elements $a_0, a_1, \ldots\,, a_k \in \F_p^*$ and
integers $0 = t_0 < t_1 < \ldots < t_k<p$,
the multiplicity of any root $\rho$ of the polynomial
$$
\sum_{i=0}^k a_iX^{t_i} \in \F_p[X]
$$
is at most $k$.
\end{lemma}

\begin{proof} Let
$$
F(X) = \sum_{i=0}^k a_iX^{t_i}.
$$
Then for the $j$ derivative $F^{(j)}(X)$ we have
$$
F^{(j)}(X)X^j = \sum_{i=0}^k\prod_{h=0}^{j-1} (t_i -h)  a_iX^{t_i}
$$
(where as usual, we set $F^{(0)}(X)= F(X)$).
Thus, if $r \ne 0$ is a root of multiplicity at least $k+1\le t_k
<p$  in the algebraic closure of $\F_p$,  then
$$
F^{(j)}(r) = 0, \qquad j =0, \ldots, k.
$$
Therefore, the homogeneous system of equations
$$
 \sum_{i=0}^k \prod_{h=0}^{j-1} (t_i -h) x_i  = 0, \qquad j =0, \ldots, k,
$$
has a non-zero solution $x_i =  a_ir^{t_i}$, $i =0, \ldots, k$.
This implies
$$
\det\left[ \( \prod_{h=0}^{j-1} (t_{i} -h) \)_{i,j=0, \ldots, k}\right] = 0,
$$
which is impossible for $0 = t_0 < t_1 < \ldots < t_k <p$ as an easy
calculation shows that
$$
\det\left[ \( \prod_{h=0}^{j-1} (t_{i} -h) \)_{i,j=0, \ldots, k}\right]
 =  \prod_{0 \le i <j\le k}
    (t_j - t_i) \ne 0.
$$
The above contradiction implies the desired result.
\end{proof}

\section{Domination Number of a Graph}
\label{sec:dom}

Let $G=(V,E)$ be a simple undirected graph of order $n$.
A {\it dominating set\/}
$S$ of $G$ is a vertex subset such that any vertex of $V\setminus
S$  has a neighbour in $S$. Intuitively, a dominating set of a
graph is a vertex subset whose neighbours, along with themselves,
make up the vertex set of the graph.

The minimum cardinality of a dominating set of $G$ is called the
{\it domination number\/} $\gamma(G)$ of $G$. In other words,
$$\gamma(G)=\min_{S\subseteq V(G)}\left\{|S|~:~V(G)\subseteq
\bigcup_{v\in S} \hat{N}(v)\right\},$$
where $\hat{N}(v)$ denotes the closed neighbourhood of a vertex
$v$.

We denote by $\delta(G)$ the minimum degree of $G$.

When $\delta(G)$ is big enough, there are very good upper bounds
for the domination number of the graph $G$ in terms of $\delta(G)$
and $n$ (see, for example,~\cite{CSSF, HHS2}). However, for small values
of $\delta(G)$ the classical result of Ore~\cite{ORE} is stronger
and provides an upper bound for the domination number of a graph
with no isolated vertices:

\begin{lemma} 
\label{lem:Dom Set}
If $G$ is a graph of order $n$ with
$\delta(G)\geq1$, then
$$\gamma(G)\leq \frac{n}{2}.
$$
\end{lemma}

\section{Proof of Theorem~\ref{thm:split p}}

Since $p > t_k$,  by
Lemma~\ref{lem:SprEq Mult}
the
multiplicity  of each non-zero root of a polynomial of the
form~\eqref{eq:LacPoly} does not exceed $k$.
Hence, if  a polynomial $F(X) \in \F_p[X]$ of the form~\eqref{eq:LacPoly} splits
completely  over $\F_p$ then the equation
$$
a_0 + a_{1}x^{t_1} + \ldots + a_{n}x^{t_k} = 0, \qquad x \in \F_p^*,
$$
with $1 \le t_1< \ldots < t_k$
has at least $t_k/k$ solutions.  Then, from Lemma~\ref{lem:SprEq Zeros} we have
$$
t_k/k  =  O\(p^{1 -1/k} D_\vec{t}^{1/k}\),
$$
where
$$
D_\vec{t}  = \min_{0 \le i \le n} \max_{j \ne i} \gcd(t_j - t_i, p-1).
$$
Thus $D_\vec{t}t \mid p-1$ and, since $k$ is fixed,
\begin{equation}
\label{eq:Dt}
t\ge D_\vec{t} \gg  t_k^{k}p^{-(k-1)} = t^{k}p^{-(k-1)}.
\end{equation}
We now fix $D \mid p-1$, and for each $\vec{t} = (t_1,\ldots,
t_k)$ construct a graph $G_\vec{t}(D)$ on $k+1$ vertices $0,
\ldots, k$, connecting $i$ and $j$ if and only if $\gcd(t_i-t_j,
p-1) \ge D$ (where, as before $t_0=0$).

Clearly, if $D_\vec{t}=D$ and  $G_\vec{t}(D)=G$ then $\delta(G)\ge
1$.

Now, for a fixed positive integer  $D \le  t < p$ and a graph $G$
with $k+1$ vertices and $\delta(G)\ge 1$, we estimate the number $M_p(D,G,t)$ of vectors
$\vec{t} = (t_1,\ldots, t_k) \in \Z^k$ with $1 \le t_1 < \ldots
<t_k$ and $t_k=t$ such that $G_\vec{t}(D)=G$.
 Summing over all
graphs $G$ (since $k$ is fixed
there are only finitely many graphs) and admissible values of $D$,
that is, with $t\ge D  \gg    t^{k}p^{-(k-1)}$, see~\eqref{eq:Dt},
 leads to the desired estimate.

Given a graph $G$ with $k+1$ vertices and $\delta(G)\ge 1$,
we now fix a dominating set $S$ in $G$ of cardinality $\# S  =
\fl{(k+1)/2}$, which exists by Lemma~\ref{lem:Dom Set} (obviously,
we can always add more vertices to $S$ if necessary to guarantee
$\# S  = \fl{(k+1)/2}$). So for each $j \not \in S$ with $j \ne 0,
k$, there is $i\in S$ such that $\gcd(t_i-t_j, q-1) \ge D$. So if
$t_i$ is fixed, then $t_j$ can take at most
\begin{equation}
\label{eq:Dom Choice}
\sum_{\substack{d\mid p-1\\ d \ge D}} \frac{t}{d}  \ll \frac{t}{D}
\sum_{d\mid p-1}1  = \frac{t}{D} p^{o(1)}
\end{equation}
values, where we have used the known bound on the divisor
function, (see~\cite[Theorem~320]{HardyWright}). Finally, when
$t_k=t$ is fixed, each $t_i$, $i \in S$, can take at most $t$
values.

Furthermore,
if both $0,k \in S$ then there are only
$$\# S - 2 \le \fl{(k+1)/2}-2 = \fl{(k-3)/2}$$
elements $t_i$ with $i \in S\setminus \{0,k\}$ to be chosen.
After all values of $t_i$ with $i \in S$ are fixed, we see
from~\eqref{eq:Dom Choice} that the remaining
$$k+1 - \#S =  \rf{(k+1)/2}$$
elements $t_j$, $j \not \in S$,  can be chosen in at most
$(tp^{o(1)}/D)^{\rf{(k+1)/2}}$ ways.
So in this case
\begin{equation}
\label{eq:MDt20} M_p(D,G,t) \le t^{\fl{(k-3)/2}}
(t/D)^{{\rf{(k+1)/2}}}p^{o(1)} = t^{k-1}
D^{-\rf{(k+1)/2}}p^{o(1)}.
\end{equation}

If $0 \in S$ but $k\not \in S$, or  $0 \not \in S$ but $k \in S$, then the same
argument implies:
\begin{equation}
\label{eq:MDt11} M_p(D,G,t) \le t^{\fl{(k-1)/2}}
(t/D)^{{\rf{(k-1)/2}}}p^{o(1)} = t^{k-1}
D^{-\rf{(k-1)/2}}p^{o(1)}.
\end{equation}

Finally, if both  $0,k \not \in S$ then we get
\begin{equation}
\label{eq:MDt02} M_p(D,G,t) \le t^{\fl{(k+1)/2}}
(t/D)^{{\rf{(k-3)/2}}}p^{o(1)} = t^{k-1}
D^{-\rf{(k-3)/2}}p^{o(1)}.
\end{equation}

Clearly, bound~\eqref{eq:MDt02} dominates the
bounds~\eqref{eq:MDt20} and~\eqref{eq:MDt11}. In particular, for
$t \ge D \gg t^{k}p^{-(k-1)}$ we obtain
$$
M_p(D,G,t) \le t^{k-1 - k\rf{(k-3)/2} }
p^{(k-1)\rf{(k-3)/2}+o(1)}.
$$
Since, as we have mentioned,  there are only finitely many possibilities for the graphs
$G_\vec{t}(D)$, recalling~\eqref{eq:Dt} and the bound on
the divisor function (see~\cite[Theorem~320]{HardyWright}),
 we obtain the desired result.

\section{Comments}

A slight modification of our
approach can easily produce a nontrivial bound for $1\le k \le 3$
as well, however we do not know how to relax the condition~\eqref{eq:large t}.

It is certainly an interesting question to show that almost all
$k$-lacunary polynomials of a large degree are irreducible over
$\F_p$. In fact, as a first step one can try to get a lower bound
on the degree over $\F_p$ of the splitting field of a ``random''
$k$-lacunary polynomial.

\section*{Acknowledgements}

The authors would like to thank the referee for the careful reading
of the manuscript and helpful suggestions.

During the preparation of this work the second author was supported in part by
the  Australian Research Council  Grant~DP1092835.

\end{document}